\documentclass[smallextended,referee,envcountsect]{svjour3} 
\smartqed 
\journalname{JOTA}

\usepackage{amssymb,amsmath}
\allowdisplaybreaks[4]
\usepackage{float}
\usepackage{color,fullpage}
\usepackage{pifont}
\usepackage[utf8]{inputenc}

\newcommand{\RR}{\mathbb{R}}

\newcommand{\prox}{{\mathbf{prox}}}

\newcommand{\cF}{{\mathcal{F}}}

\newcommand{\cS}{{\mathcal{S}}}

\begin{document}

\title{A New Exact Worst-Case Linear Convergence Rate of the Proximal Gradient Method}

%\subtitle{Using  the  LaTex Template}

\author{Xiaoya Zhang   \and Hui Zhang  %etc.
}


\institute{Xiaoya Zhang,  Hui Zhang \at
              Department of Mathematics, National University of Defense Technology, Changsha, 410073, Hunan, China. \\
              \email{zhangxiaoya09@nudt.edu.cn, h.zhang1984@163.com.}           %  \\
             }

\date{Received: date / Accepted: date}

\maketitle

\begin{abstract}
In this note,  we establish a new exact worst-case linear convergence rate of the proximal gradient method in terms of the proximal gradient norm,  which complements the recent results in \cite{taylor2018exact} and implies a refined descent lemma.  Based on the new lemma, we improve 
  the linear convergence rate of the objective function accuracy under the Polyak-{\L}ojasiewicz inequality.% which is strictly weaker than the strong convexity.
\end{abstract}
\keywords{linear convergence   \and proximal gradient method  \and  strongly convex \and Polyak-{\L}ojasiewicz inequality}
\subclass{90C25 \and 90C22 \and 90C20}

%All acknowledgements should be placed in the back of the paper after Conclusions..
\section{Introduction}
A well-known algorithm for minimizing  the sum of a smooth function with a non-smooth convex one is  the proximal gradient (PG) method.  Recently, the authors of \cite{taylor2018exact} studied the exact worst-case linear convergence rates of the PG method for three different standard performance measures: objective function accuracy, distance to optimality and residual gradient norm. However, the first and third measures rely on the minimizers and the optimal value, which are in general unknown; while the second measure is usually difficult to compute. 
On the other hand, the  proximal gradient (also called stepsize in \cite{drusvyatskiy2016nonsmooth}) norm is suggested as a more appealing stopping criteria in \cite{drusvyatskiy2016nonsmooth}.
This motivates us to consider the proximal gradient norm as an alternative to the existing three performance measures. 
%In \cite{drusvyatskiy2016nonsmooth})  the authors  showed  that the residual measure operator can be reliably used to terminate the proximal algorithm algorithm.

As a result, we derive an exact worst-case linear convergence rate for the PG method in terms of the proximal gradient norm. The proof idea shares the same spirit of Theorem 2 in  \cite{nutini2018active} but is quite different from that in \cite{taylor2018exact}.  Our result not only complements the recent results in \cite{taylor2018exact}, but also helps us refine the classic descent lemma for the PG method and further yields an improved linear convergence rate of the objective function accuracy for non-strongly convex case. 

%We summarize that:
%\begin{enumerate}
%\item We  study the theoretical exact worst-case linear convergence rates of  the proximal gradient norm  for PG method, which complements the results  in \cite{taylor2018exact}.  But the proof is new and quite different from that in  \cite{taylor2018exact}.
%\item Based on the first conclusion,  we propose a new refined descent lemma and  derive better linear convergence rate for solving non-strongly convex problem. 
%\end{enumerate}

%The paper is organized as follows: We first recall some basic definitions and results in Section 2. Some main results of proximal gradient algorithm are discussed in Section 3.  We conclude the note with some final remarks in Section 4.

\section{Notations and preliminaries}
\subsection{Notations and definitions}
Throughout the paper, $\RR^n$ will denote an $n$-dimensional Euclidean space associated with inner-product $\langle \cdot, \cdot \rangle$ and induced norm $\| \cdot \|$. For any nonempty $S \subset \RR^n$, we define the distance function by $d(x, S) := \inf_{y \in S} \|x - y\|$. Besides, we define the indicator function of a set $C \subset \RR^n$ as
$$ \iota_C: C \rightarrow [-\infty, +\infty]: x \rightarrow  \left \{
\begin{aligned} 
0, \qquad & x \in C; \\
+\infty, \qquad & otherwise.
\end{aligned}
\right.$$
Recall some basic notions,  the domain of the function $f : \mathbb{R}^d \rightarrow (-\infty,+\infty]$ is defined by $\text{dom}~f = \{x \in \mathbb{R}^d: f(x)< +\infty\}$. We say that $f$ is proper if $\text{dom}~f \neq \emptyset$. 

The $L$-smoothness and $\mu$-strongly convexity are defined as:
\begin{description}
\item [$\Box$] \emph{ $L$-smoothness:} $\forall x \in \RR^n, \| \nabla f(x) - \nabla f(y)\| \le L \|x-y\|$ holds.
\item [$\Box$] \emph{$\mu$-strong convexity:}  $ f(x) - \frac{\mu}{2} \|x\|^2$  is convex on $\RR^n$.
%or  \\*
%$\forall x, y \in \RR^n, f(\alpha x + (1- \alpha) y) \le \alpha f(x) + (1-\alpha)f(y) - \frac{\mu}{2}\alpha (1−\alpha)\|x-y\|^2$ holds for any $\alpha \in [0,1]$.
\end{description}

For simplicity, we make the following notations:
\begin{description}
\item[$\bullet$] $\cF_L^{1,1}(\RR^n)$:  the class of $L$-smooth convex functions from $\RR^n$ to $\RR$;
\item[$\bullet$]  $\cS_{\mu,L}^{1,1}(\RR^n)$:  the class of $L$-smooth and $\mu$-strongly convex functions from $\RR^n$ to $\RR$;
\item[$\bullet$]  $\Gamma_0(\RR^n)$: the class of proper closed and convex functions from $\RR^n$ to $(-\infty, +\infty]$.
\end{description}
Obviously, we have $ \cS_{\mu,L}^{1,1}(\RR^n) \subseteq  \cF_L^{1,1}(\RR^n)$. 

\subsection{The proximal gradient algorithm}
In this note, we consider  the composite convex minimization:
\begin{align}
\min_{x \in \RR^n} \{ \varphi (x) : = f(x) + g(x) \} \label{Eq:P}
\end{align}
where $f \in \cF_L^{1,1}(\RR^n)$ and $g \in \Gamma_0(\RR^n)$. 

We focus on the PG method with constant step size $t$ to solve (\ref{Eq:P}).
For simplicity, we use the superscript "+" to denote the subsequent iterate. The PG method can be simply expressed by
$$x^+=\prox_{tg}(x-t\nabla f(x))=x-t\cdot\mathcal{G}_t(x), t>0$$
where $\prox_{tg}(x):=   \arg \min_{u \in \RR^n} \left\{  t g(u) + \frac{1}{2}\|u-x\|^2 \right \}$ and  $\mathcal{G}_t(x)=t^{-1}\left(x-\prox_{tg}(x-t\nabla f(x))\right)$ is defined as the proximal gradient.
By the equality $\prox_{tg}=(I+ t\partial g)^{-1}$, we have
$x-t\nabla f(x)\in x^++ t\partial g(x^+),$
which implies that there exists $s^+\in \partial g(x^+)$ such that
$$x^+=x-t(\nabla f(x)+s^+).$$

\subsection{Two important lemmas}
%Before presenting the main result of the paper, we first introduce the following two lemmas.
Our analysis will rely on the following two lemmas.
\begin{lemma}[Theorem 2.1.12, \cite{nesterov2013introductory}; Theorem 4, \cite{taylor2017smooth}]\label{Pro:1}
If $f\in \cS^{1,1}_{\mu,L}(\RR^n)$, then for any $x, y\in \RR^n$ we have
$$\mu\|x-y\|\leq \|\nabla f(x)-\nabla f(y)\|\leq L \|x-y\|,$$ 
and
$$\langle \nabla f(x)-\nabla f(y), x-y \rangle \geq \frac{\mu L}{\mu+ L}\|x-y\|^2+\frac{1}{\mu+L}\|\nabla f(x)-\nabla f(y)\|^2,$$
and the smooth strongly convex interpolation formula
$$f(x)\geq f(y)+\langle \nabla f(y), x-y\rangle+ \frac{1}{2L} \|\nabla f(x)-\nabla f(y)\|^2+\frac{\mu L}{2(L-\mu)}\|x-y-\frac{1}{L}(\nabla f(x)-\nabla f(y))\|^2.$$
\end{lemma}

\begin{lemma}[Theorem 3.5, \cite{drusvyatskiy2018error}]\label{Pro:2}
Let $\varphi =f +g $, where $f \in\cF^{1,1}_L(\RR^n)$ and $g\in\Gamma_0(\RR^n)$. For any $x\in \RR^n$, it holds that
$$\|\mathcal{G}_t(x)\|\leq d(0, \partial \varphi(x)).$$
\end{lemma}

\section{Main result and implications}
In this section, we present two new results for the PG method: one is an exact worst-case linear convergence rate in terms of the proximal gradient norm, and the other is a refined sufficient decrease property of the objective function value.
\subsection{Main result}
Now, we are ready to present the main result of this note.
\begin{theorem}\label{addlem}
Let $\varphi =f +g $, where $f\in\cS^{1,1}_{\mu,L}(\RR^n)$ and $g\in\Gamma_0(\RR^n)$. Denote $\rho(t):=\max\{ |1-Lt|, |1-\mu t|\}$.  Then,  the PG method for minimizing $\varphi$ achieves the exact worst-case linear convergence rate in terms of the proximal gradient norm:
\begin{equation}\label{add10}
\|\mathcal{G}_t(x^+)\|\leq d(0, \partial \varphi(x^+))\leq\rho(t)\|\mathcal{G}_t(x)\|\leq \rho(t) d(0, \partial \varphi(x)).
\end{equation}
In particular, for $f\in\cF^{1,1}_{L}(\RR^n)$, $g\in\Gamma_0(\RR^n)$, and $0< t\leq \frac{2}{L}$, it holds that
$$\|\mathcal{G}_t(x^+)\|\leq d(0, \partial \varphi(x^+))\leq \|\mathcal{G}_t(x)\| \leq d(0, \partial \varphi(x)).$$
\end{theorem}
\begin{proof}
Note that $s^+ \in \partial g(x^+)$ and hence $d(0,\partial \varphi(x^+)) \le \| \nabla f(x^+) + s^+\|$. Therefore, to show \eqref{add10}, it suffices to show that $\|\nabla f(x^+)+s^+\|^2\leq \rho^2(t)\|\mathcal{G}_t(x)\|^2$ in view of Lemma \ref{Pro:2}.  Using Lemma \ref{Pro:1}, we derive that
\begin{align*}
&\|\nabla f(x^+)+s^+\|^2 \\
= & \|\nabla f(x)+s^+ +\nabla f(x^+)-\nabla f(x)\|^2  \\
= &\|\nabla f(x)+s^+\|^2+2\langle \nabla f(x)+s^+, \nabla f(x^+)-\nabla f(x) \rangle +  \|\nabla f(x^+)-\nabla f(x)\|^2\\
= &\frac{1}{t^2}\|x^+-x\|^2-\frac{2}{t} \langle x^+-x, \nabla f(x^+)-\nabla f(x) \rangle +  \|\nabla f(x^+)-\nabla f(x)\|^2\\
\leq &\frac{1}{t^2}\|x^+-x\|^2-\frac{2}{t}\left(\frac{\mu L}{\mu+ L}\|x^+-x\|^2+\frac{1}{\mu+L}\|\nabla f(x^+)-\nabla f(x)\|^2\right) +  \|\nabla f(x^+)-\nabla f(x)\|^2\\
=& \frac{1}{t^2}\left[(1-\frac{2t\mu L}{\mu+ L})\|x^+-x\|^2 + t(t- \frac{2}{\mu+L})\|\nabla f(x^+)-\nabla f(x)\|^2 \right]\\
 \leq &\frac{1}{t^2}\left[(1-\frac{2t\mu L}{\mu+ L})\|x^+-x\|^2 + t\max\{L^2(t- \frac{2}{\mu+L}), \mu^2(t- \frac{2}{\mu+L})\}\|x^+-x \|^2 \right]\\
=&\frac{1}{t^2}\max\{1-\frac{2t\mu L}{\mu+ L} + tL^2(t- \frac{2}{\mu+L}), 1-\frac{2t\mu L}{\mu+ L}  +  t\mu^2(t- \frac{2}{\mu+L}) \}\|x^+-x \|^2 \\
=&\frac{1}{t^2} \max\{ (1-Lt)^2, (1-\mu t)^2\}\|x^+-x \|^2=\rho^2(t)\|\mathcal{G}_t(x)\|^2.
\end{align*}
\end{proof}
Here, the factor $\rho(t)$ can not be improved; otherwise, it will contradict the following exact worst-case convergence rate, which was recently established in \cite{taylor2018exact}:
$$\|\nabla f(x^+)+s^+\|^2\leq \rho^2(t)\|\nabla f(x)+s\|^2, ~~\forall s\in \partial g(x).$$

\subsection{Implicated result}
The second result is a refined  version of the classic descent lemma(see \cite[Corollary 2.2.1]{nesterov2013introductory}\cite[Lemma 2.3]{beck2009fast}).
%descent lemma, guaranteeing sufficient decrease of  the objective function value. Concept of the descent lemma appears in many papers(see \cite[Lemma 2]{zhang2017proximal}\cite[Lemma 4.1]{teboulle2018simplified}).  Here, we improve the classic descent lemma for the PG method. 

\begin{lemma}\label{addlem1}
Let $\varphi = f + g $, where $f  \in \cS^{1,1}_{\mu,L}(\RR^n)$ and $g\in\Gamma_0(\RR^n)$. Then,  the PG method for minimizing $\varphi$ has the refined sufficient decrease property
\begin{equation}\label{add20}
\varphi(x)\geq \varphi(x^+) +\frac{t}{2}\|\mathcal{G}_t(x)\|^2 + \frac{t}{2(1-\mu t)}\|\mathcal{G}_t(x^+)\|^2, 0< t\leq \frac{1}{L}.
\end{equation}
In particular, 
\begin{itemize}
\item for $f \in \cF^{1,1}_{L}(\RR^n)$, $g \in \Gamma_0(\RR^n)$, it holds that
\begin{equation}\label{add21}
\varphi(x)\geq \varphi(x^+) +\frac{t}{2}\|\mathcal{G}_t(x)\|^2 + \frac{t}{2}\|\mathcal{G}_t(x^+)\|^2, 0< t\leq \frac{1}{L}.
\end{equation}
\item  for $f \in \cF^{1,1}_{L}(\RR^n)$, $g \equiv 0$, it holds that
\begin{equation}\label{add22}
f(x) \geq f(x^+) +\frac{t}{2}\|\nabla f(x)\|^2 + \frac{t}{2}\|\nabla f(x^+)\|^2, 0< t \leq \frac{1}{L}.
\end{equation}
\end{itemize}
\end{lemma}
\begin{proof}
Note that $0<t \leq L^{-1}$ implies $t^{-1}\geq L$ and the fact that $\cS^{1,1}_{\mu,L}(\RR^n) \subset\cS^{1,1}_{\mu, t^{-1}}(\RR^n)$. We can use the smooth strongly convex interpolation formula with $L=t^{-1}$ and $y=x^+$ in Lemma \ref{Pro:1} to get
$$f(x)\geq f(x^+)+\langle \nabla f(x^+), x-x^+\rangle+ \frac{t}{2} \|\nabla f(x)-\nabla f(x^+)\|^2+\frac{\mu }{2(1-\mu t)}\|x-x^+-t(\nabla f(x)-\nabla f(x^+))\|^2.$$
The convexity of $g$ gives $g(x)\geq g(x^+)+\langle s^+, x-x^+\rangle$ since $s^+ \in \partial g(x^+)$. Adding these two inequalities, we derive that
\begin{align*}
\varphi(x)\geq  & \varphi(x^+)+\langle \nabla f(x^+)+s^+, x-x^+\rangle+ \frac{t}{2} \|\nabla f(x)-\nabla f(x^+)\|^2\\
&+\frac{\mu }{2(1-\mu t)}\|x-x^+-t(\nabla f(x)-\nabla f(x^+))\|^2\\
= &\varphi(x^+)+\langle \nabla f(x)+s^+, x-x^+\rangle - \langle \nabla f(x^+)- \nabla f(x), x^+ -x\rangle \\
&+ \frac{t}{2} \|\nabla f(x)-\nabla f(x^+)\|^2 +\frac{\mu }{2(1-\mu t)}\|x-x^+-t(\nabla f(x)-\nabla f(x^+))\|^2
\end{align*}
Using the expression $x^+=x-t(\nabla f(x)+s^+)$, we can further derive that
\begin{align*}
\varphi(x) \geq  &  \varphi(x^+)+\frac{1}{t}\|x-x^+\|^2 - \langle \nabla f(x^+)- \nabla f(x), x^+ -x\rangle \\
&+ \frac{t}{2} \|\nabla f(x)-\nabla f(x^+)\|^2 +\frac{\mu t^2}{2(1-\mu t)}\|s^+ + \nabla f(x^+)\|^2\\
=& \varphi(x^+) + \frac{1}{2t}\|t(\nabla f(x^+)- \nabla f(x))-x^+ +x\|^2\\
&+\frac{1}{2t}\|x-x^+\|^2 +\frac{\mu t^2}{2(1-\mu t)}\|s^+ + \nabla f(x^+)\|^2\\
=& \varphi(x^+)+\frac{1}{2t}\|x-x^+\|^2  +\frac{t}{2(1-\mu t)}\|s^+ + \nabla f(x^+)\|^2.
\end{align*}
Note that $x-x^+=t \mathcal{G}_t(x)$ and
$$\|s^+ + \nabla f(x^+)\|\geq d(0, \partial \varphi(x^+))\geq \|\mathcal{G}_t(x^+)\|.$$
We finally obtain
$$\varphi(x)\geq \varphi(x^+) +\frac{t}{2}\|\mathcal{G}_t(x)\|^2 + \frac{t}{2(1-\mu t)}\|\mathcal{G}_t(x^+)\|^2.$$
This completes the proof.
\end{proof}

\begin{remark}
In  \cite[Corollary 2.2.1]{nesterov2013introductory}, for $\varphi = f+ g$ with $f \in \cS^{1,1}_{\mu,L}(\RR^n)$ and $g$ being the indicator function of a set $Q$, the descent lemma of the projected gradient method can be stated as
\begin{equation}\label{compare}
\varphi(x) \geq \varphi(x^+) +\frac{t}{2}\|g_Q(x,t)\|^2, 0 < t \le \frac{1}{L}.
\end{equation}
where $g_Q(x, t) := t^{-1}(x-x^+)$ is the gradient mapping of $f$ on $Q$.

In \cite[Lemma 2.3]{beck2009fast}, for $\varphi = f+g$ with $f  \in \cF^{1,1}_{L}(\RR^n)$ and $g \in \Gamma_0(\RR^n)$, the corresponding descent lemma of the PG method is:
\begin{equation}\label{compare}
\varphi(x) \geq \varphi(x^+) +\frac{L}{2}\|x^+ - x\|^2.
\end{equation}
Remarkably, our result improves these existing descent  lemmas.
\end{remark}

%Although the exact wort-case linear convergence rate cannot be get for the objective function accuracy, it is still meaningful, especially when $f$ is not strongly convex. 
With the refined descent lemma, we can show a better linear convergence rate in terms of the objective function accuracy for the gradient descent method under the classic Polyak-{\L}ojasiewicz inequality \cite{polyak1963gradient}\cite{lojasiewicz1963topological}. 
%and the PG method without the assumption of strong convexity, compared with the existing results\cite{karimi2016linear}. 
%Our results are based on a classic Polyak-{\L}ojasiewicz inequality \cite{polyak1963gradient} and a generalized Polyak-{\L}ojasiewicz inequality(corresponds to res-obj-EB  in \cite[Definition 5]{zhang2019new}), which are strictly weaker than the strong convexity.
%We first derive a result for the gradient descent method.

\begin{corollary}
Let $f \in\cF^{1,1}_{L}(\RR^n)$, $g \equiv 0$.  Assume that $f$ satisfies the Polyak-{\L}ojasiewicz inequality for some $\eta > 0$:
\begin{align*}
\forall x \in \text{dom}~ f, \frac{1}{2}\|\nabla f(x)\|^2 \geq \eta (f(x) -\min f).
\end{align*}
Let $x^+ = x - t \nabla f(x)$, $0< t \leq \frac{1}{L}$, then it holds that
\begin{align}\label{add23}
f(x^+) -\min f  \le  \frac{1-\eta t}{1+ \eta t} (f(x) - \min f).
\end{align}
\end{corollary}
\begin{proof}
Using the Polyak-{\L}ojasiewicz inequality and  (\ref{add22}) in Lemma \ref{addlem1}, we have
\begin{align*}
f(x) &\ge f(x^+) +\frac{t}{2}\|\nabla f(x)\|^2 + \frac{t}{2}\|\nabla f(x^+)\|^2 \\ 
&\ge f(x^+)  + \eta t( f(x) -\min f) +  \eta t( f(x^+) - \min f) .
\end{align*}
Rearranging and subtracting $ \min f$ from both sides yield
$$  f(x^+) -\min f  \le  \frac{1-\eta t}{1+ \eta t} (f(x) - \min f) .$$
\end{proof}
\begin{remark}
%In our corollary, the linear convergence rate of  the objective function accuracy could be $\frac{L-\eta}{L+ \eta}$(if we set $t = \frac{1}{L}$),
%whereas previous results were only establishing a 􏰀$(1- \frac{\eta}{L})$ 􏰁 linear convergence rate (see \cite[Theorem 1]{karimi2016linear}) under the same condition. 
The result (\ref{add23}) with $t = \frac{1}{L}$ improves the existing linear convergence rate in \cite[Theorem 1]{karimi2016linear} from $(1-\frac{\eta}{L})$ to $\frac{L-\eta}{L+ \eta}$.
\end{remark}

%Furthermore, we extend and prove a better linear convergence rate for the PG method under a generalized Polyak-{\L}ojasiewicz inequality. This generalization is intuitive.
Finally, we extend the result above to the PG method.
\begin{corollary}
Let $f \in\cF^{1,1}_{L}(\RR^n)$, $g\in\Gamma_0(\RR^n)$. Assume that $\varphi = f + g$ satisfies the generalized Polyak-{\L}ojasiewicz inequality for some $\eta > 0$:
\begin{align*}
\forall x \in \text{dom}~ \varphi , \frac{1}{2}\|\mathcal{G}_t(x)\|^2 \geq \eta (\varphi(x) -\min \varphi).
\end{align*}
Let $x^+ = x - t \mathcal{G}_t(x)$, $0< t \leq \frac{1}{L}$, then it holds that
\begin{align}
\varphi(x^+) -\min \varphi  \le  \frac{1-\eta t}{1+ \eta t} (\varphi(x) - \min \varphi).
\end{align}
\end{corollary}
\begin{proof}
Using the generalized Polyak-{\L}ojasiewicz inequality and  (\ref{add21}) in Lemma \ref{addlem1}, we have
\begin{align*}
\varphi(x) &\ge \varphi(x^+) +\frac{t}{2}\|\mathcal{G}_t(x)\|^2 + \frac{t}{2}\|\mathcal{G}_t(x^+)\|^2 \\ 
&\ge \varphi(x^+)  + \eta t( \varphi(x) -\min \varphi) +  \eta t( \varphi(x^+) - \min \varphi) .
\end{align*}
Rearranging and subtracting $ \min \varphi$ from both sides give us 
$$  \varphi(x^+) -\min \varphi  \le  \frac{1-\eta t}{1+ \eta t} (\varphi(x) - \min \varphi) .$$
\end{proof}

%In \cite{zhang2019new},  Hui Zhang  provided a unified framework that allows us to find new connections between many existing error bound conditions. Thus, from other  error bound conditions, we can also obtain the similar result .

%\section{Conclusions}
%In this note, inspired by  \cite{taylor2018exact}, we establish the exact worst-case linear convergence rates of  the residual measure operator for proximal gradient algorithm based on a new and simple proof. 
%Besides, we present a refine descent lemma. Based on this lemma,  we prove a better convergence rate of the objective function accuracy with  Polyak-{\L}ojasiewicz(PL) inequality, which is strictly weaker than strong convexity.

\section*{Acknowledgements}
This work is supported by the National Science Foundation of China (No.61571008).

%\bibliographystyle{unsrt}
%\bibliography{ref}

\end{document}